\documentclass[12pt,reqno]{amsart}
\usepackage{amsmath,amsthm,amsfonts,amssymb}
\usepackage[french,english]{babel}
\usepackage{tikz}
\usepackage{psfrag,pslatex}
\usepackage{graphicx}
\usepackage{epsfig}

\numberwithin{equation}{section}

\newcommand{\dist}{\operatorname{dist}}

\newcommand{\e} {\varepsilon}

\newcommand{\la} {\lambda}      

\newcommand{\vfi}{\varphi}

\newcommand{\s} {\sigma}

\def \NN {{\mathbb N}}

\def \Z {{\mathbb Z}}
\def \T {{\mathbb T}}

\def \cK {{\mathcal K}}

\def \cU {{\mathcal U}}

\def \cW {{\mathcal W}}

\newtheorem{theorem}{Theorem}[section]
\newtheorem{corollary}[theorem]{Corollary}
\newtheorem{proposition}[theorem]{Proposition}

\newtheorem{conjecture}[theorem]{Conjecture}

\newtheorem{definition}[theorem]{Definition}

\theoremstyle{remark}
\newtheorem{remark}[theorem]{Remark}

\begin{document}

\renewcommand{\subjclassname}{\textup{2000} Mathematics Subject Classification}


\setcounter{tocdepth}{2}



\renewcommand{\subjclassname}{\textup{2000} Mathematics Subject Classification}


\title{On the three-legged accessibility property}

\begin{abstract} We show that certain types of the three-legged accessibility property of a partially hyperbolic diffeomorphism imply the existence of a unique minimal set for one strong foliation and the transitivity of the other one.
In case the center dimension is one, we also give a criteria to obtain three-legged accessibility in a robust way. We show some applications of our results to the time-one map of Anosov flows, skew products and certain Anosov diffeomorphisms with partially hyperbolic splitting.

\end{abstract}

\thanks{}

\author{Jana Rodriguez Hertz}
\address{Department of Mathematics, Southern University of Science and Technology of China, 1088 Xueyuan Rd., Xili, Nanshan District, Shenzhen, Guangdong, China 518055}
\email{rhertz@sustc.edu.cn}

\author{Ra\'ul Ures}
\address{Department of Mathematics, Southern University of Science and Technology of China, 1088 Xueyuan Rd., Xili, Nanshan District, Shenzhen, Guangdong, China 518055} \email{ures@sustc.edu.cn}

\maketitle
\section{Introduction}

A diffeomorphism $f$ of a closed manifold $M$ is {\em partially
hyperbolic} if the tangent bundle $TM$ of $M$, splits into three invariant sub-bundles: 
$TM=E^{s}\oplus E^{c}\oplus E^{u}$
such that  all unit vectors
$v^\s\in E^\s_x$ ($\s= s, c, u$) with $x\in M$ satisfy  :
\begin{equation}\label{pointwise.ph}
 \|T_xfv^s\| < \|T_xfv^c\| < \|T_xfv^u\| 
\end{equation}
for some suitable Riemannian metric. The {\em stable bundle} $E^{s}$ must also satisfy
$\|Tf|_{E^s}\| < 1$ and the {\em unstable bundle}, $\|Tf^{-1}|_{E^u}\| < 1$. The bundle $E^{c}$ is called {\em center bundle}.\par
It is a well-known fact that the {\em strong} bundles, 
$E^s$ and $E^u$, are uniquely integrable \cite{bp,hps}.  That is, there are invariant strong foliations ${\mathcal W}^{s}$ and ${\mathcal W}^{u}$ tangent, respectively to the invariant bundles $E^{s}$ and $E^{u}$. 
(However, the integrability of $E^{c}$ is a more delicate matter)\par
In this paper we will deal with partially hyperbolic diffeomorphisms that satisfy certain type of accessibility property. Recall that such a diffeomorphism $f$ satisfies the accessibility property (we will also say that $f$ is accessible) if any pair of points can be joined by  curve that is piecewise tangent to either $E^s$ or $E^u$. We will say that $f$ has the {\em three-legged accessibility property} if in the definition of accessibility you can choose the curve joining each pair of points consisting of three arcs tangent to either $E^s$ or $E^u$ and with uniformly bounded length. Moreover, we will say that $f$ is {\em $sus$-accessible} if it satisfies the  three-legged accessibility property and the three-legged curve of the definition can be chosen, for all pair of points,  in such a way that the first arc is stable, the second unstable and the last one stable. The $usu$-accessibility property is defined in a analogous way.

These more restrictive accessibility properties impose some limitations to the strong foliations of a partially hyperbolic diffeomorphism.  Our first result is in this direction and states the following:
\begin{theorem}\footnote{The authors jointly with Federico Rodriguez Hertz already had a proof of this result more than ten years ago.} \label{mintrans} Let $f\in C^{1}$ be an $usu$-accessible partially hyperbolic diffeomorphism. Then,
\begin{enumerate}
\item  ${\mathcal W}^{u}$ has a unique minimal set. 
\item ${\mathcal W}^{s}$ is transitive, that is, it has a dense leaf. 
\end{enumerate}
\end{theorem}

The importance of the uniqueness of the minimal sets of ${\mathcal W}^{u}$ lies in the fact that it imposes important obstructions to the presence of attractors, in fact, under certain conditions they are unique. 
We will establish this result in the following theorem.

\begin{theorem}\label{ugibbs} Let $f$ be a $C^{1+\alpha}$ partially hyperbolic diffeomorphism and assume that $\cW^u(f)$ has a unique minimal set $\Lambda$. If there is an $u$-Gibbs measure $\mu$ supported on $\Lambda\,$ for which all center Lyapunov exponents are negative then, $\mu$ is the unique $u$-Gibbs measure for $f$ and, as a consequence, it is SRB.

\end{theorem}

Before showing some applications of Theorem \ref{mintrans}  we give some necessary conditions to obtain $sus$-accessibility. 
\begin{theorem}\label{robustness}
Let $f\in C^{1}$ be an accessible partially hyperbolic diffeomorphism with one-dimensional center such that ${\mathcal W}^{u}$ is minimal. Then, there is an open $C^1$-neighborhood $\cU$ of $f$ such that every $g\in\cU$ is $usu$-accesible.

\end{theorem}

Of course we have the analogous result if ${\mathcal W}^{s}$ is minimal. In fact, in their pioneer paper about stable ergodicity \cite{gps}, Grayson, Pugh and Shub showed that the time-one map $\varphi$ of the geodesic flow of a closed surface of constant negative curvature is both $sus$- and $usu$-accessible. Dolgopyat \cite{dolgo} used this property to get some consequences about the uniqueness of SRB measures in a neighborhood of $\varphi$. In Section \ref{anosovflows}
we show that most time-one maps of  Anosov flows, including the case of geodesic flows, are both $sus$- and $usu$-accessible. This generalizes Burns, Pugh and Wilkinson \cite{bpw} result about the accessibility of Anosov flows. Also, as an application of Theorem \ref{mintrans} we obtain a generalization of the main result of Bonatti and Guelman \cite{bg} about the approximation of time-one maps of  Anosov flows by Axiom A diffeomorphisms.

We also give some applications to skew products (Section \ref{skewproducts}) and to Anosov diffeomorphisms (Section \ref{anosovdiffeomorphisms}). In the case of skew products we show that our results can be applied to an open and dense set of isometric circle extensions of volume-preserving partially hyperbolic diffeomorphisms.   In the case of Anosov diffeomorphisms, we answer positively the following conjecture of Gogolev, Maimon and Kolgomorov \cite{gmk}.  Let $A$ be a hyperbolic automorphism of the 3-torus with three real eigenvalues $|\lambda_1|<1<|\lambda_2|<|\lambda_3|$.

\begin{conjecture}\label{strong_trans}
For all analytic diffeomorphisms $f$
in a sufficiently small neighborhood of $A$ the strong unstable
foliation $\cW^{uu}$
is transitive, i.e., it has a dense leaf.
\end{conjecture}

In fact, we get a stronger result. We obtain the transitiveness of $\cW^{uu}$ for any 3-dimensional $C^1$ Anosov diffeomorphism $f$ with a partially hyperbolic splitting.

\section{Preliminaries}\label{section.preliminaries}

\subsection{Partial hyperbolicity}
\begin{definition} A diffeomorphism $f$ of a closed manifold $M$ is {\em partially
hyperbolic} if the tangent bundle $TM$ of $M$, splits into three invariant sub-bundles: 
$TM=E^{s}\oplus E^{c}\oplus E^{u}$
such that  all unit vectors
$v^\s\in E^\s_x$ ($\s= s, c, u$) with $x\in M$ satisfy  :
\begin{equation}\label{pointwise.ph}
 \|T_xfv^s\| < \|T_xfv^c\| < \|T_xfv^u\| 
\end{equation}
for some suitable Riemannian metric. The {\em stable bundle} $E^{s}$ must also satisfy
$\|Tf|_{E^s}\| < 1$ and the {\em unstable bundle}, $\|Tf^{-1}|_{E^u}\| < 1$. The bundle $E^{c}$ is called {\em center bundle}.

We also call $E^{cu}=E^c\oplus E^u$ and $E^{cs}=E^c\oplus E^s$

\end{definition}

It is a well-known fact that the {\em strong} bundles, 
$E^s$ and $E^u$, are uniquely integrable \cite{bp,hps}.  That is, there are invariant strong foliations ${\mathcal W}^{s}(f)$ and ${\mathcal W}^{u}(f)$ tangent, respectively,  to the invariant bundles $E^{s}$ and $E^{u}$. 
(However, the integrability of $E^{c}$ is a more delicate matter) In general,  we will call ${\mathcal W}^{\sigma}(f)$  any foliation tangent to $E^\sigma$, $\sigma=s,u, c, cs, cu$, whenever it exists and $W^\sigma_f(x)$ the leaf of ${\mathcal W}^{\sigma}(f)$ passing through $x$. A subset $\Lambda$ is {\em $\sigma$-saturated} if $W^\sigma_f(x)$ for every $x\in \Lambda$. A closed $\sigma$-saturated subset $K$ is {\em minimal} if $\overline{W^\sigma_f(x)}=\Lambda$ for every $x\in \Lambda$.  We say that a foliation is minimal if $M$ is a minimal set for it. A foliation is {\em transitive} if it it has a dense leaf. \par

\subsection{Skew products} 
In this subsection we consider skew products  for which the base is a volume-preserving Anosov diffeomorphism and the fibers are circles (also known as isometric circle extensions) 
That means that we have $F_\vfi:M\times \mathbb S^1\to M\times \mathbb S^1$ such that $F_\vfi(x, \theta)= (f(x), R_{\vfi(x)}\theta)$ where $f:M\to M$ is a $C^r$ volume-preserving Anosov diffeomorphisms, $R_{\alpha}$ is the rotation of angle $\alpha$ and $\vfi:M\to\mathbb S^1$ is a $C^r$ map. $r\geq2$. These diffeomorphisms are partially hyperbolic, see \cite{bp}. Since $f$ is volume preserving we have that $F_\vfi$ preserves the measure given by the product of the volume of $M$ and the Lebesgue measure on $\mathbb S^1$ (the Haar measure of $\mathbb S^1$ seen as a Lie group)

The following is proved in \cite{bp} (see also Proposition 2.1 in \cite{bw})
The center bundle of $F_\vfi$ is tangent to the circle fibers and the strong stable and strong unstable bundles have the same dimension as the corresponding bundles of the base Anosov diffeomorphism. In particular, it has  a center foliation and its leaves are the circle fibers. $F_\vfi$ is dynamically coherent, that means that there are  invariant foliations $\cW^{c\sigma}(F_\vfi)$ tangent  tangent to $E^{c\sigma}$, $\sigma = s, u$. Each leaf of $\cW^{c\sigma}(F_\vfi)$ is the preimage under $\pi$ of a leaf of $\cW^{\sigma}(f)$, $\sigma= s, u$ ($\pi:M\times \mathbb S^1\to M$ is the projection on the first coordinate)  In particular, any leaf of $\cW^{c\sigma}(F_\vfi)$ is dense and it is the product of a leaf of $\cW^{\sigma}(f)$ and $\mathbb S^1$ again $\sigma=s,u$. Each leaf of a strong foliation is 
 is the graph of a $C^r$ function from the corresponding  leaf of the corresponding strong foliation of $f$ to $\mathbb S^1$.

\subsection{$u$-Gibbs measures}
We will only give a very brief introduction to the $u$-Gibbs measures with the purpose of doing this paper as self-contained as possible. For a more complete presentation we refer to \cite{bdp, dolgo2}.

\begin{definition} Let $f$ be a $C^{1+\alpha}$ partially hyperbolic diffeomorphism. An $f$-invariant probability measure $\mu$ is $u$-Gibbs if it admits conditional measures along local strong unstable leaves which are absolutely continuous with respect to Lebesgue. 

\end{definition}

The condition that $f$ is $C^{1+\alpha}$ is needed in order to get the absolute continuity of the strong unstable foliation.

The densities of the conditional measures are bounded away from zero and infinity and then, the support of a $u$-Gibbs measure consists of entire strong unstable leaves.  The set of $u$-Gibbs measures is a convex compact subset of the set of probability measures. 

Recall that an {\em SRB measure} is an invariant probability measure such that admits conditional measures along unstable manifolds which are absolutely continuous with respect to Lebesgue. Observe the difference with $u$-Gibbs measures for which these conditional measures are taken along {\bf strong} unstable leaves. SRB measures of partially hyperbolic diffeomorphisms are always $u$-Gibbs measures. In general the converse is not true but if $f$ admits a unique $u$-Gibbs measure $\mu$ then, $\mu$ is SRB (see, for instance, \cite[Theorem 11.16]{bdp})

\section{Proofs of the main results}\label{proofs}

\subsection{Proof of Theorem \ref{mintrans}}

\begin{proof}[Proof of a)]
Let $A,\,B\subset M$ be two closed $u$-saturated sets, $a\in A$, $b\in B$ and $K$ a bound of the length of the curves given by the definition of $usu$-accessibility. 
Let $n\in \NN$ and $\gamma$ a curve of length less than $K$ joining $f^{-n}(a)$ and $f^{-n}(b)$ consisting of three arcs $\gamma_i$ $i=1,2,3$ such that $\gamma_1$ and $\gamma_3$ are tangent to $E^u$ and $\gamma_2$ is tangent to $E^s$. Observe that, since $A$ and $B$ are $u$-saturated, $\gamma_1\subset f^{-n}(A)$ and $\gamma_3\subset f^{-n}(B)$. On the one hand, the previous observation gives that $f^n(\gamma_2)$ is an arc joining $A$ and $B$. On the other hand,  since $\gamma_2$ is tangent to $E^s$, we have that the length of  $f^n(\gamma_2)$ is less than $K \lambda^n$ for some uniform $\lambda<1$. Thus, the distance between $A$ and $B$ is less than $K \lambda^n$ for all $n\in \NN$. This implies that $A\cap B\neq \emptyset$. We have that any pair of closed $u$-saturated subsets has nonempty intersection and since the intersection of $u$-saturated sets is an $u$-saturated set we obtain that the family of all closed $u$-saturated sets satisfies the FIP. This implies that $S=\bigcap \{A\subset M; \,\,A \text{ is closed and $u$-saturated}\}\neq \emptyset$. It is not difficult to see that $S$ is the unique minimal set of ${\mathcal W}^{u}$. This ends the proof of the first part of Theorem \ref{mintrans}.
\end{proof}
\begin{remark} Observe that since $f$ is a diffeomorphism and ${\mathcal W}^{u}$ is $f$-invariant we have that the minimal set given by Theorem \ref{mintrans} is $f$-invariant.

\end{remark}

\begin{proof}[Proof of b)]
Let $U,\, V\subset M$ be to open sets. It is enough to show that there is a leaf of  ${\mathcal W}^{s}$ that has nonempty intersection with both $U$ and $V$. 

Take $x\in U$, $y\in V$ and $n\in \NN$. Then, there is  curve $\gamma$ of length less than $K$ joining $f^{n}(x)$ and $f^{n}(y)$ consisting of three arcs $\gamma_i$ $i=1,2,3$ such that $\gamma_1$ and $\gamma_3$ are tangent to $E^u$ and $\gamma_2$ is tangent to $E^s$. Since $\gamma_1$ and $\gamma_3$ are tangent to $E^u$ we have that the length of $f^{-n}(\gamma_i)$ $i=1,3$ is less than $K\lambda^n$ for for some uniform $\lambda<1$. Observe that $x$ and $y$ are extreme points of the arcs $f^{-n}(\gamma_1)$ and $f^{-n}(\gamma_3)$ respectively and the other two extremes are in the same stable leaf. Since $K\lambda^n\to 0$
we have that for $n$ large enough $f^{-n}(\gamma_1)\subset U$ and $f^{-n}(\gamma_3)\subset V$. Then, there is a stable leaf intersecting simultaneously $U$ and $V$. This finishes the proof of the second part of Theorem \ref{mintrans}.
\end{proof}

\subsection{Proof of Theorem \ref{robustness}}

In order to proof Theorem \ref{robustness} we need to introduce a new definition. In this subsection we will assume that $\dim(E^c)=1$.
\begin{definition}
We say that two unstable disks $U_1$ and $U_2$ are skew iff  
\begin{enumerate}
\item There is a dimension $\dim(E^u)+1$ disk $D$ containing $U_2$.
\item $U_2$ separates $D$ into two connected components.
\item $U_1$ and $D$ define the holonomy map  $h_s:U_1 \to D$ as $h_s(x)=W^s_\e(x)\cap D$ for $\e$ small.
\item $h_s(U_1)$ intersects both connected components of $D\setminus U_2$

\end{enumerate}
\end{definition}

\begin{remark}\label{crossopeness}

\begin{enumerate}
\item Being skew is an open condition in the following sense: if we have two disks $U'_1$ and $U'_2$ close enough to $U_1$ and $U_2$ they are also skew and this remains true if $U'_1$ and $U'_2$ are unstable disks of a diffeomorphism $g$ close enough to $f$.
\item If $U_1$ and $U_2$ are skew there exist $x_i\in U_i$ $i=1,2$ such that $x_2\in W^s_\e(x_1)$.
\end{enumerate}
\end{remark}

\begin{proof}[Proof of Theorem \ref{robustness}]

It is already known that accessibility implies the existence of two unstable disks $U_i$ $i=1,2$ that are skew (see \cite{di, hhu})
Given neighborhoods $V_i$ of $U_i$ $i=1,2$  there exists $C>0$ such that $W^u_C(x)$ intersect both $V_1$ and $V_2$. This implies that given any pair of 
points $x_i\in M$ $i=1,2$, $W^u_C(x_i)$ contains a disk in $V_i$ for $i=1,2$. The continuos dependence of the unstable foliation on the diffeomorphisms imply that the same is true, with the same constant $C$, for any $g$ close enough to $f$ (maybe we have to take $V_i$ $i=1,2$ a little bit larger) So, the two disks are skew (see Remark \ref{crossopeness}) and have two points $y_i$ $i=1,2$ that are joined by a stable curve of length less than $\e$. The previous considerations imply that $y_i$ con be joined to $x_i$ by an unstable arc of length less than $C$, $i=1,2$. That means that $g$ is $usu$-accessible for all $g$ $C^1$ close enough to $f$ with $K=2C+\e$ and finishes the proof of the theorem.
\end{proof}

\subsection{Proof of Theorem \ref{ugibbs}}

\begin{proof}
Let $F$ be a leaf of $\cW^u(f)$. Since $\Lambda$ is the unique minimal set of $\cW^u(f)$ we have that $\Lambda \subset \overline F $. Since $\mu$ is supported on $\Lambda$ and all its center Lyapunov exponents are negative, the absolute continuity of the stable partition implies that $F$ has a positive Lebesgue measure set of points that are in the basin of $\mu$. This clearly implies that $\mu$ is the unique $u$-Gibbs measure for $f$.

\end{proof}

\section{Anosov flows}\label{anosovflows}

In this section we will apply our results to the time-one maps of Anosov flows.
On the one hand, Burns, Pugh and Wilkinson \cite{bpw} (see also \cite{di}) proved that if $\vfi$ is the time-one map of a transitive Anosov flow, accessibility is equivalent to the fact of $E^s\oplus E^u$ not being integrable. On the other hand, Plante \cite{pl} showed that in case $E^s\oplus E^u$ is integrable the flow is topologically equivalent to a suspension of a hyperbolic diffeomorphism. Moreover, in dimension three, joint integrability implies that the flow {\bf is} a suspension. 
These comments lead to the following result. 

\begin{theorem}\label{anosovflows}
Let  $\vfi_t$ be  a transitive Anosov flow and assume that $E^s\oplus E^u$ is not integrable. Then, there is a $C^1$-neighborhood $\cU$ of $\vfi=\vfi_1$ such that the strong unstable and the strong stable  foliations of any $g\in \cU$ are transitive and have a unique minimal set. 
\end{theorem}

\begin{proof} Since the Anosov flow is transitive and $E^s\oplus E^u$ is not integrable we have that $\vfi$ is accessible and both strong foliations are minimal. Since $\vfi$ is partially hyperbolic with  one dimensional center we can apply Theorem \ref{robustness} that immediately implies the thesis.
\end{proof} 

In particular, any $g\in \cU$ can have at most one transitive hyperbolic attractor (in fact it would be topologically mixing) and one transitive hyperbolic repeller. Then, we obtain  as  corollary the following generalization of the main result of \cite{bg}.

\begin{corollary}
If the time-one map of a transitive Anosov flow $\vfi_t$ is $C^1$-approximated by
diffeomorphisms having more than one transitive hyperbolic attractor, then $E^s\oplus E^u$ is integrable. In particular,  $\vfi_t$ is topologically equivalent to
the suspension of a hyperbolic diffeomorphism. Moreover, if the dimension of the ambient manifold is three then, $\vfi_t$ is the suspension of an Anosov diffeomorphisms.
\end{corollary}

\begin{proof}
Since an attractor is a compact $u$-saturated set it necessarily contains the unique minimal set given by Theorem \ref{anosovflows}. 
Since transitive hyperbolic attractors are disjoint $g\in \cU$ can have only one hyperbolic attractor. This proves the corollary.
\end{proof}

In fact, to obtain the thesis of this corollary the only thing we need is the robustness of the uniqueness of the minimal set of the strong unstable foliation. Then, the same proof yields the following theorem. 

\begin{theorem}\label{uniqueattract}
Let $f$ be a stably $usu$-accessible partially hyperbolic diffeomorphism. Then, there is a $C^1$ neighborhood $\cU$ of $f$ such that every $g\in \cU$ has at most one transitive hyperbolic attractor. 

\end{theorem}

Observe that the uniqueness of the minimal set of $\cW^u(f)$ implies that the transitive hyperbolic attractor, if it exists, is in fact topologically mixing. 

In case the diffeomorphisms are $C^{1+\alpha}$ the previous results are consequence of Theorem \ref{ugibbs}.

\section{Skew products}\label{skewproducts}

In this section we consider skew products (also known as isometric extensions) for which the base is a volume-preserving Anosov diffeomorphism and the fibers are circles. 
That means that we have $F_\vfi:M\times \mathbb S^1\to M\times \mathbb S^1$ such that $F_\vfi(x, \theta)= (f(x), R_{\vfi(x)}\theta)$ where $f:M\to M$ is a $C^r$ volume-preserving Anosov diffeomorphisms, $R_{\alpha}$ is the rotation of angle $\alpha$ and $\vfi:M\to\mathbb S^1$ is a $C^r$ map. $r\geq2$. These diffeomorphisms are partially hyperbolic and their ergodic theory  is well known. Burns and Wilkinson \cite{bw} proved the following (in fact they prove a much stronger result)

\begin{theorem}[\cite{bw}]\label{access} The set of accessible isometric $\mathbb S^1$-extensions of a volume-preserving Anosov diffeomorphism is open and dense in the $C^r$-topology.
\end{theorem}

We will show that if a skew product as above satisfies the accessibility property, both the strong stable and the strong unstable foliation are minimal. The arguments to show this fact are inspired in similar arguments of Plante \cite{pl}.

Observe that $F_\vfi$ commutes with $G_\alpha:M\times \mathbb S^1\to M\times \mathbb S^1$, $G_\alpha(x, \theta)=(x, R_\alpha \theta)$ $\forall \alpha \in \mathbb S^1$. This implies that $\cW^\sigma(F_\vfi)$ is $G_\alpha$-invariant for $\sigma=s,\,u$. 

\begin{theorem}\label{accmin} Let $F_\vfi$ be as above and suppose that it satisfies the accessibility property. Then, $\cW^\sigma(F_\vfi)$ is minimal, $\sigma=s,\,u$. 

\end{theorem}

In fact, the previous theorem is a consequence of the following stronger fact.

\begin{proposition} Let $F_\vfi$ be as above and suppose that $\cW^\sigma(F_\vfi)$ is not minimal. Then, $E^s\oplus E^u$ is integrable.

\end{proposition}

\begin{proof} Suppose that there is $x\in M\times \mathbb S^1$ such that  $W^u_{F_\vfi}(x)$ is not dense. Since $\overline{W^u_{F_\vfi}(x)}$ is $u$-saturated we can take $K\subset \overline{W^u_{F_\vfi}(x)}$  a minimal subset. Call $\pi:M\times \mathbb S^1\to M$ to the projection onto $M$, that is $\pi(y,\theta)=y$. It is not difficult to see that $\pi(W_{F_\vfi}^u(z))=W^u_f(\pi(z))$. In particular, since $f$ is transitive and $K$ is compact and $u$-saturated we have that $\pi(K)=M$. The strategy is to prove that $K$ is also $s$-saturated. This will immediately imply that $F_\vfi$ is not accessible (moreover, the invariance of the strong foliations under $G_\alpha$ $\forall \alpha \in \mathbb S^1$ will imply that $E^s\oplus E^u$ is integrable)

Our first claim is that $\# (K\cap \{y\}\times\mathbb S^1)$ is finite for every $y\in M$. Suppose that it is false, then there is $y_0\in M$ such that for every $\e>0$, $S=\{y_0\}\times \mathbb S^1$ has two points $(y_0, \theta_1)$ and $(y_0, \theta_2)$ at distance less than $\e$. This means that  $\alpha=\theta_2-\theta_1<\e$. Observe that  $G_\alpha(y_0, \theta_1)=(y_0, \theta_2)$ and then, $G_\alpha(K)\cap K\neq\emptyset$. Since $K$ is minimal and $\cW^u(F_\vfi)$ is $G_\alpha$-invariant we have that $G_\alpha(K)=K$. Then, $G_{n\alpha}(y_0, \theta_1)\in K$ $\forall n \in \Z$. In particular, $K$ is $\e$-dense in $S$. Since $\e$ is arbitrary we obtain that $S\subset K$. This implies that $W^u_{F_\vfi}(S)\subset K$ and it is not difficult to see that $W^u_{F_\vfi}(S)$ is dense, then $K=M\times \mathbb S^1$ contradicting that $W^u_{F_\vfi}(x)$ is not dense.

The previous considerations allow us to define $\Psi: M\to \NN$ as $\Psi(y)= \#(K\cap \{y\}\times\mathbb S^1)$. Our second claim is that this function is upper semicontinuous. Suppose it is false. Then, there is a sequence $(y_n)_n\subset M$ such that $y_n\to y$ and $\Psi(y_n)\to \eta >\Psi(y)$. Since $K$ is compact, we also have that  $\lim (K\cap \{y_n\}\times\mathbb S^1)\subset K\cap \{y\}\times\mathbb S^1$. In particular, we have that for any $\e>0$ and $n$ large enough there two points in  $\{y_n\}\times\mathbb S^1$ at distance less than $\e$. An argument similar to the one used in the proof of our first claim gives that $\{y\}\times\mathbb S^1\subset K$ and then, $K=M\times \mathbb S^1$, a contradiction.

Now we want to proof that $\Psi$ is constant. On the one hand, observe that if $y'\in W^u_f(y)$, $\Psi(y')=\Psi(y)$. On the other hand, since $\Psi$ is semicontinuous it has continuity points and  it is locally constant at this kind of points. Then, the minimality of the unstable foliation of $f$ implies that $\Psi$ is equal to a constant $h$.

The previous considerations imply that $(K, \pi)$ is an $h$-fold covering of $M$ and it is not difficult to see that $\cK=\cup_{\alpha\in \mathbb S^1} G_\alpha(K)$ is a $C^0$ foliation of  $M\times \mathbb S^1$ with compact leaves homeomorphism to $K$. 
Since $F_\vfi$ is an isometric extension $\dist_c (K_1, K_2)=\min \{dist(k_1, k_2); \, k_1\in K_1, \, k_2\in K_2, \, \pi(k_1)=\pi(k_2)\}=\dist_c(F_\vfi(K_1), F_\vfi(K_2))$ where $K_1,K_2\in \cK$. In particular, $\dist(F^n_\vfi(K_1), F^n_\vfi(K_2)$ does not go to zero. Then, any stable manifold intersects only one leaf of $\cK$.  That means that the leaves of $\cK$ are $s$-saturated and then, $\cK$ is a foliation tangent to $E^s\oplus E^u$.
\end{proof}

We get the following corollary of Theorems \ref{access} and \ref{robustness} for an open and dense set of skew products over volume-preserving Anosov diffeomorphisms.

\begin{corollary}
Let $F$ be an accessible isometric circle extension of a volume-preserving Anosov diffeomorphism. Then, there is a $C^1$ neighborhood $\cU$ of $F$ such that:
\begin{itemize}
\item $\cW^\sigma(F)$ has a unique minimal set, $\sigma=s,u$.
\item $\cW^\sigma(F)$ is transitive, $\sigma=s,u$.
\end{itemize}

\end{corollary}

Moreover, we can apply Theorem \ref{ugibbs} and Theorem \ref{uniqueattract} to these diffeomorphisms and obtain the corresponding conclusions.

\begin{remark}
Hammerlindl and Potrie \cite{hp} have shown that partially hyperbolic diffeomorphisms on 3-nilmanifolds (different than the torus) always are $sus$- and $usu$-accessible. In particular we can also apply Theorem \ref{ugibbs} and Theorem \ref{uniqueattract} to them. On the other hand, Shi \cite{sh} has announced that, on 3-nilmanifolds,  there are partially hyperbolic diffeomorphism satisfying Axiom A. Of course, they have only one attractor and one repeller.

\end{remark}

\section{Anosov diffeomorphisms}\label{anosovdiffeomorphisms}

In this section we will suppose that $A$ is a hyperbolic automorphism of $\T^3$ with three real eigenvalues $|\lambda_1|< 1<|\la_2|<|\la_3|$. Along this section we will assume  that $f$ is an Anosov diffeomorphism isotopic to $A$ and such that it admits a partially hyperbolic splitting $E^s\oplus E^c\oplus E^u$. Observe that $E^c\oplus E^u$ corresponds to the unstable bundle of the hyperbolic splitting of $f$. We obtain the following result that answers Conjecture \ref{strong_trans} by Gogolev, Maimon and Kolmorgorov. 

\begin{theorem}
Let $f$ be as above. Then, the strong unstable foliation of $f$ is transitive. 

\end{theorem}

\begin{proof} Call $h$ to the conjugacy between $f$ and $A$. Observe that since the strong stable manifolds of $f$ are its stable manifolds we have that $h$ sends strong stable manifolds of $f$ into strong stable manifolds of $A$. Thus, the strong stable foliations of $f$ is minimal.

Ren, Gan and Zhang \cite{rgz} have proved that, in our setting, the following statements are equivalent:

\begin{itemize}

\item $f$ is not accessible.

\item $E^s\oplus E^u$ is integrable.

\item $h$ sends strong unstable manifolds of $f$ into strong unstable manifolds of $A$.

\end{itemize}

Then, on the one hand, if $f$ is not accessible we have that $\cW^u(f)$ is minimal and, in particular, it is transitive. 
On the other hand, if $f$ is accessible, since we already know that $\cW^s(f)$ is minimal, we can apply Theorem \ref{robustness} (with the roles of $u$ and $s$ reversed) We get the transitivity of $\cW^u(f)$ by applying Theorem \ref{mintrans}.
This ends the proof of the theorem.

\end{proof}


\end{document}